\documentclass[12pt]{amsart}
\usepackage{graphicx}
\usepackage[active]{srcltx}
\usepackage{amsthm,mathrsfs}
\usepackage{a4wide}
\usepackage{amsfonts}
\usepackage{amsmath}
\usepackage{amssymb}
\newtheorem{lemma}{Lemma}
\newtheorem{theorem}{Theorem}
\newtheorem{corollary}{Corollary}

\usepackage{float}
\usepackage{url}
\usepackage{enumitem}
\usepackage{epstopdf}

\newcommand {\p} {\mathbb{P}}

\newcommand {\ve} {\varepsilon}

\makeatletter\def\blfootnote{\xdef\@thefnmark{}\@footnotetext}\makeatother
\allowdisplaybreaks
\title[Metric results on the discrepancy of sequences $\left(a_{n}
\alpha\right)_{n \geq 1}$]{\bf Metric results on the discrepancy 
of sequences $\left(a_{n} \alpha\right)_{n \geq 1}$ modulo one for integer
sequences $\left(a_{n}\right)_{n \geq 1}$ of polynomial growth}

\author{Christoph Aistleitner} 
\address{ Institute of Financial Mathematics and applied Number Theory,
University Linz}
\email{aistleitner@math.tugraz.at}

\author{Gerhard Larcher} 
\address{ Institute of Financial Mathematics and applied Number Theory,
University Linz}
\email{gerhard.larcher@jku.at}

\thanks{The first author is supported by a Schr\"odinger scholarship of the
Austrian
Research Foundation (FWF). The second author is supported by the Austrian
Science Fund (FWF): Project F5507-N26, which is part of the Special Research
Program "Quasi-Monte Carlo Methods: Theory and Applications"}

\begin{document}

\begin{abstract}
An important result of H. Weyl states that for every sequence
$\left(a_{n}\right)_{n \geq 1}$ of distinct positive integers the sequence of
fractional parts of $\left(a_{n} \alpha \right)_{n\geq 1}$ is uniformly
distributed modulo one for almost all $\alpha$. However, in general it is a very hard
problem to calculate the precise order of convergence of the discrepancy of $\left(\left\{a_{n} \alpha\right\}\right)_{n \geq 1}$ for almost all
$\alpha$. In particular it is very difficult to give sharp lower bounds for the
speed of convergence. Until now this was only carried out for lacunary sequences
$\left(a_{n}\right)_{n \geq 1}$ and for some special cases such as the
Kronecker sequence $\left(\left\{n \alpha\right\}\right)_{n \geq1}$ or the sequence $\left(\left\{n^2 \alpha\right\}\right)_{n \geq1}$. In the
present paper we answer the question for a large class of sequences
$\left(a_{n}\right)_{n \geq 1}$ including as a special case all polynomials $a_{n} =
P\left(n\right)$ with $P \in \mathbb{Z} \left[x\right]$ of degree at least 2.
\end{abstract}

\date{}
\maketitle

\section{Introduction} \label{sect_1}

In the present paper we will study distribution properties of sequences $\left(x_{n}\right)_{n \geq
1}$ of the form $x_{n}=\left\{a_{n} \alpha\right\}$ in the unit interval, where
$\alpha$ is a given real and $\left(a_{n}\right)_{n \geq 1}$ is a given sequence
of integers, and where $\{\cdot\}$ denotes the fractional part. In particular we
are interested in the behavior of the discrepancy $D_{N}$ of these sequences
from a metric point of view.\\

The discrepancy $D_{N}$ of a sequence $\left(x_{n}\right)_{n \geq 1}$ in
$\left[\left. 0,1\right)\right.$ is given by
$$
D_{N} = \underset{0 \leq a < b \leq 1}{\sup} \left|\frac{A_{N}
\left(\left[\left.a,b\right.\right)\right)}{N} - \left(b-a\right)\right|
$$
where $A_{N} \left(\left[\left. a,b\right.\right)\right) = \# \left\{1 \leq n
\leq N \left| \right. x_{n} \in \left[\left. a,b\right.\right)\right\}$. The
sequence $\left(x_{n}\right)_{n \geq 1}$ is uniformly distributed in
$\left[0,1\right]$ if and only if $\underset{N \rightarrow \infty}{\lim} D_{N} =
0$. It was shown by H. Weyl~\cite{wey} that for every sequence
$\left(a_{n}\right)_{n \geq 1}$ of distinct positive integers the sequence
$\left(\left\{a_{n} \alpha\right\}\right)_{n \geq 1}$ is uniformly distributed
for almost all $\alpha$. In~\cite{baker} R.C. Baker gave a corresponding
discrepancy estimate: Let $\left(a_{n}\right)_{n \geq 1}$ be a strictly
increasing sequence of positive integers. Then for almost all $\alpha$ for the
discrepancy $D_{N}$ of $\left(\left\{a_{n} \alpha\right\}\right)_{n \geq1}$ we
have
\begin{equation} \label{baker}
ND_{N} = \mathcal{O} \left(N^{\frac{1}{2}} \left(\log
N\right)^{\frac{3}{2}+\varepsilon}\right)
\end{equation}
for all $\varepsilon > 0$. It is known that this result is best possible, up to
logarithmic terms, as a generic result covering \emph{all} strictly increasing
integer sequences $(a_n)_{n \geq 1}$ (see~\cite{bpt}). However, in most cases of
interesting particular sequences $(a_n)_{n \geq 1}$ no metric lower bounds for
the discrepancy of $\left(\left\{a_{n} \alpha\right\}\right)_{n \geq 1}$ are
known at all, and in these cases it is totally unclear how close the generic upper bound
in~\eqref{baker} comes to the ``correct'' metric order of the discrepancy.\\

One case where the precise metric asymptotic order of the discrepancy of
$\left(\left\{a_{n} \alpha\right\}\right)_{n \geq 1}$ is known is the case when $\left(a_{n}\right)_{n \geq 1}$ is a lacunary sequence, i.e., if
$\frac{a_{n+1}}{a_{n}} \geq 1+\delta$ for some constant $\delta > 0$. For this
case, where certain independence properties of the sequence $\left(\left\{a_{n}
\alpha\right\}\right)_{n \geq 1}$ can be used,  W. Philipp~\cite{philipp} proved
that for almost all $\alpha$
\begin{equation} \label{phil}
\frac{1}{4 \sqrt{2}} \leq \underset{N \rightarrow \infty}{\lim \sup}
\frac{ND_{N}}{\sqrt{2 N \log \log N}} \leq c_\delta
\end{equation}
holds. Even more precise results were obtained by K. Fukuyama~\cite{fu2}, taking
into account the number-theoretic structure of $\left(a_{n}\right)_{n \geq 1}$.
It is interesting to note that the asymptotic result in~\eqref{phil} is in
accordance with the Chung--Smirnov law of the iterated logarithm, which
prescribes a discrepancy of order $(N \log \log N)^{-1/2}$ for the discrepancy
of a ``random'' sequence in the unit interval, almost surely. \\

For the case $a_{n} =n$, that is for the Kronecker sequence, it follows from
results of Khintchine~\cite{khin} in the metric theory of continued fractions
that
\begin{equation} \label{kh1}
ND_{N} = \mathcal{O} \left(\log N \left(\log \log
N\right)^{1+\varepsilon}\right)
\end{equation}
and
\begin{equation} \label{kh2}
ND_{N} = \Omega \left(\log N \log \log N\right)
\end{equation}
hold for every $\varepsilon > 0$ for almost all $\alpha$. Of course the same
result also holds in the case when $a_{n}$ is a polynomial of degree 1 in $n$.\\

For the case $a_{n} =n^2$ it follows from a result of Fiedler, Jurkat and K\"orner~\cite{fjk} that
$$
ND_{N} = \Omega \left(N^{\frac{1}{2}} (\log N )^{\frac{1}{4}}\right)
$$
holds.\\

Apart from the case of lacunary $\left(a_{n}\right)_{n \geq 1}$, the
classical case of the pure Kronecker sequence $\left(\left\{n \alpha
\right\}\right)_{n \geq 1}$, and from the example $\left(\left\{n^2 \alpha
\right\}\right)_{n \geq 1}$ only for a few further examples the precise metric
order of the discrepancy of $\left(\left\{a_{n} \alpha\right\}\right)_{n \geq1}$
is known. A very interesting special example was given recently in~\cite{AHL}.
Here for the first time an example for $\left(a_{n}\right)_{n \geq 1}$ was given
where the metric order of $ND_{N}$ is strictly between $N^{\varepsilon}$ and
$N^{\frac{1}{2}-\ve}$. Take for $\left(a_{n}\right)_{n \geq 1}$ the sequence of
integers with an even sum of digits in base 2. Then for almost all $\alpha$ we
have
$$
ND_{N}=\mathcal{O} \left(N^{\kappa+\varepsilon}\right)
$$
and
$$
ND_{N} = \Omega \left(N^{\kappa-\varepsilon}\right)
$$
for all $\varepsilon > 0$, where $\kappa$ is a constant of the form
$\kappa=0,404\ldots$. Interestingly, the precise value of the constant $\kappa$
is still unknown; for details see~\cite{AHL} and~\cite{foumau}.\\

Until now, however, to the best of our knowledge, nothing was known on the exact
metric order of the discrepancy of the sequences $\left(\left\{n^{k}
\alpha\right\}\right)_{n \geq 1}$ for $k \ge 3$ or related polynomial sequences, apart from
the general upper bound in~\eqref{baker}. In particular, it seems that no
non-trivial lower bound whatsoever was known in this case.\\

In this paper we will show that for a large class of sequences
$\left(a_{n}\right)_{n \geq 1}$ the discrepancy $D_{N}$ of $\left(\left\{a_{n} \alpha\right\}\right)_{n \geq
1}$ has an asymptotic order of roughly $\frac{1}{\sqrt{N}}$. This class in
particular contains all sequences $a_{n} = P(n)$ with $P\left(n\right) \in
\mathbb{Z} \left[x\right]$ of degree larger or equal 2. Consequently, together
with the results from~\eqref{kh1} and~\eqref{kh2}, we have now a fairly complete
understanding of the metric discrepancy behavior of sequences generated by
polynomials with integer coefficients.\\

Our main result is the following:

\begin{theorem} \label{th_1}
Let $P \in \mathbb{Z} \left[x\right]$ be a polynomial of degree $d \geq 2$ and
let $\left(m_{n}\right)_{n \geq 1}$ be an arbitrary sequence of pairwise
different integers with $\left|m_{n}\right| \leq n^{t}$ for some $t \in
\mathbb{N}$ and all $n \geq n(t)$. Then for the discrepancy $D_{N}$ of the
sequence $\left(\left\{P\left(m_{n}\right)\alpha \right\}\right)_{n \geq 1}$ we
have for almost all $\alpha$
$$
ND_{N} \geq N^{\frac{1}{2}-\varepsilon}
$$
for all $\varepsilon >0$ and for infinitely many $N$.
\end{theorem}

As a direct consequence of Theorem~\ref{th_1} (by choosing $m_n = n$) we obtain the following corollary.

\begin{corollary} \label{co_1}
Let $P \in \mathbb{Z} \left[x\right]$ be a polynomial of degree $d \geq 2$. Then for the discrepancy $D_{N}$ of the
sequence $\left(\left\{P(n) \alpha \right\}\right)_{n \geq 1}$ we
have for almost all $\alpha$
$$
ND_{N} \geq N^{\frac{1}{2}-\varepsilon}
$$
for all $\varepsilon >0$ and for infinitely many $N$.
\end{corollary}

The same estimate for example also holds if we choose $m_n = p_n$, the $n-$th prime, or  $m_n = [n \beta ]$ for some fixed $\beta > 1$.\\ 

Theorem~\ref{th_1} is a consequence of the more general

\begin{theorem} \label{th_2}
Let $f:\mathbb{N} \rightarrow\mathbb{Z}$ be a function with
$\left|f\left(n\right)\right| \leq n^{t}$ for some $t \in \mathbb{N}$ and all $n
\geq n\left(t\right)$. Set
$$
A_{f} (n) := \left\{\left(x,y\right) \in \mathbb{N} \times \mathbb{N} \left|
\right. f(x) + f(y) = n \right\}
$$
and
$$
\tilde{A}_{f} (n) := \left\{\left(x,y\right) \in \mathbb{N} \times \mathbb{N} \left|
\right. f(x) - f(y) = n \right\}
$$
and assume that we have $A_{f} (n) = \mathcal{O} \left(|n|^{\varepsilon}\right)$ as $|n| \to
\infty$ or $\tilde{A}_{f} (n) = \mathcal{O} \left(|n|^{\varepsilon}\right)$ as $n \to
\infty$ for all $\varepsilon >0$. Then for almost all $\alpha$ for the
discrepancy of the sequence $\left(\left\{f(n) \alpha\right\}\right)_{n \geq 1}$
we have
$$
ND_{N} \geq N^{\frac{1}{2}-\varepsilon}
$$
for all $\varepsilon > 0$ and for infinitely many $N$.
\end{theorem}

Theorem~\ref{th_2} is a consequence of the following result which implicitly
already was used in~\cite{AHL}, and which shows that under certain conditions on
$\left(a_{n}\right)_{n \geq 1}$ the metric behavior of the discrepancy of
$\left(\left\{a_{n} \alpha\right\}\right)_{n \geq 1}$ is determined by the
L1-norm of the exponential sums in $a_{n}\alpha$.

\begin{theorem} \label{th_3}
Let $\left(a_{n}\right)_{n \geq 1}$ be a sequence of integers such that for some
$t \in \mathbb{N}$ we have $\left|a_{n} \right|\leq n^{t}$ for all $n$ large
enough. Assume there exist a number $\tau \in \left(0,1\right)$ and a strictly
increasing sequence $\left(B_{L}\right)_{L \geq1}$ of positive integers with
$\left(B'\right)^{L} \leq B_{L} \leq B^{L}$ for some reals $B', B$ with $1 < B'
< B$, such that for all $\varepsilon >0$ and all $L > L\left(\varepsilon
\right)$ we have

\begin{equation} \label{equ_a}
\int^{1}_{0} \left|\sum^{B_{L}}_{n=1} e^{2 \pi i a_{n} \alpha}\right|d \alpha >
B^{\tau-\varepsilon}_{L}.
\end{equation}
Then for almost all $\alpha \in \left[ \left. 0,1\right.\right)$ for all
$\varepsilon > 0$ for the discrepancy $D_{N}$ of the sequence
$\left(\left\{a_{n} \alpha\right\}\right)_{n \geq 1}$ we have
$$N D_{N} > N^{\tau -\varepsilon}$$
for infinitely many $N$.
\end{theorem}

For the proof of Theorem~\ref{th_3} we need an auxiliary result from metric
Diophantine approximation which is of some interest on its own. We state this
result as Theorem~\ref{th_4}. In the statement of the theorem and in the sequel,
$\mathbb{P}$ denotes the one-dimensional Lebesgue measure.

\begin{theorem} \label{th_4}
Let $\left(R_{L}\right)_{L \geq 0}$ be a sequence of measurable subsets of
$\left[ \left.0,1\right)\right.$, with $\mathbb{P} \left(R_{L}\right) \geq
\frac{1}{B^{L}}$ for some constant $B \in \mathbb{R}^{+}$  and such that each
$R_{L}$ is the disjoint union of at most $A^{L}$ intervals for some $A \in \mathbb{R}^{+}$.
Then for almost all $\alpha \in \left[\left. 0,1\right)\right.$ for every
$\eta>0$ there are infinitely many integers $h_{L}$ with
$$h_{L} \leq \left(1 +\eta\right)^{L}\frac{1}{\mathbb{P} \left(R_{L}\right)}$$
and
$$\left\{h_{L} \alpha \right\} \in R_{L}.$$
\end{theorem}

Theorem~\ref{th_1} seduces to the hypothesis that a polynomial growth behavior
of $\left(a_{n}\right)_{n \geq 1}$ of degree at least 2 always implies a metric
lower bound of order $N^{\frac{1}{2}-\varepsilon}$ for the discrepancy of
$\left(\left\{a_{n} \alpha\right\}\right)_{n \geq 1}$. Here by polynomial growth
behavior of degree $d$ we mean not only that $a_{n} \geq c n^{d}$ for some $c >
0$ and all $n$ large enough, but that the stronger local condition
$$
\frac{a_{n+1}}{a_{n}} > 1 + \frac{c}{{a_{n}}^{\frac{1}{d}}}
$$
holds for some $c >0$ and all $n$ large enough.\footnote{Note that for $a_{n} =
n^{d}$ we have
$$
1+\frac{d}{{a_{n}}^{\frac{1}{d}}} < \frac{a_{n+1}}{a_{n}} < 1 +
\frac{d+1}{{a_{n}}^{\frac{1}{d}}}
$$
for all $n$ large  enough.} However, this hypothesis does not hold as is shown
in the following Theorem~\ref{th_5}.

\begin{theorem} \label{th_5}
For every integer $d \geq 1$ there is a strictly increasing sequence
$\left(a_{n}\right)_{n \geq 1}$ of integers with
$$
\frac{a_{n+1}}{a_{n}} > 1 + \frac{c}{{a_{n}}^{\frac{1}{d}}}
$$
for some $c >0$ and all $n \in \mathbb{N}$ such that for almost all $\alpha$ the
discrepancy $D_{N}$ of the sequence $\left(\left\{a_{n} \alpha\right\}\right)_{n
\geq 1}$ satisfies
$$
ND_{N} = \mathcal{O} \left(\left(\log N\right)^{2+\varepsilon}\right)
$$
for all $\varepsilon > 0$.
\end{theorem}

We conclude this section with some remarks on our theorems and some open
problems. As noted before, by the Chung--Smirnov law of the iterated logarithm for a
``random'' sequence in the unit interval the quantity $N D_N$ typically is of order
$(N \log \log N)^{1/2}$; thus, by~\eqref{baker} and by Theorems~\ref{th_1}
and~\ref{th_2}, roughly speaking the sequences investigated in the present paper
exhibit nearly random behavior. Our method only allows us to obtain discrepancy
results with an error of order $N^\ve$; it would be very interesting to improve
these error estimates to logarithmic terms, such as those in~\eqref{baker}.
\\

The fact that until now lower bounds have been almost non-existent in metric discrepancy theory comes from the fact that such lower bounds cannot be directly deduced from the second Borel--Cantelli lemma, which requires \emph{independence}. In contrast in the first Borel--Cantelli lemma, which is used to prove upper bounds, only estimates for the size of exceptional sets are required, which can be deduced from simple moment estimates and Markov's inequality. The lower bounds for the discrepancy of lacunary sequences come from an approximation of $(\{a_n \alpha\})_{n \geq 1}$ by an \emph{independent} random system, which is possible because of the fast growth of $(a_n)_{n \geq 1}$. This is not possible for slowly growing $(a_n)_{n \geq 1}$ as in the case of the present paper. Instead we have developed a method which uses estimates for the L1-norm of exponential sum plus an \emph{quasi-independence} property of the dilated functions $(\{h a_n \alpha\})$ for $h=1,2,\dots$. In the present paper the required 
lower bounds for L1-norms are deduced from upper bounds for L4-norms, which can be obtained by simply counting the number of solutions of certain linear equations (see Section~\ref{sect_3}). The quasi-independence property is established using a transition from norm-estimates for sums of dilated functions to certain sums involving greatest common divisors (GCD sums), together with recent strong bounds for such GCD sums (see Section~\ref{sect_2}). We believe that the relevance of this method goes far beyond the applications in the present paper, as the first general method for proving lower bounds in metric discrepancy theory beyond the well-known and totally different methods for lacunary sequences. We also want to point out that these topics are related to a problem in the metric theory of Diophantine approximation posed by LeVeque~\cite{LeVe} which is still unsolved; see the last section of~\cite{AHL} for details. \\

The remaining part of this paper is organized as follows. In
Section~\ref{sect_2} we first prove Theorem~\ref{th_4} and then
Theorem~\ref{th_3}. In Section~\ref{sect_3} we deduce Theorem~\ref{th_2} and
Theorem~\ref{th_1} from Theorem~\ref{th_3}. Finally in Section~\ref{sect_4} we
prove Theorem~\ref{th_5}.

\section{Proofs of Theorem~\ref{th_4} and Theorem~\ref{th_3}} \label{sect_2}

\begin{proof}[Proof of Theorem~\ref{th_4}]

Let $1_{L}\left(\alpha\right)$ denote the indicator function of the set $R_{L}$,
extended with period 1, and $\mathbb{I}_{L}(\alpha) := 1_{L} (\alpha) -
\int^{1}_{0} 1_{L} (\omega) d \omega.$\\

We have 
$$\int^{1}_{0} 1_{L} \left(\alpha\right) d \alpha = \mathbb{P}
\left(R_{L}\right)$$
and
$$\int^{1}_{0} \mathbb{I}_{L}\left(\alpha\right) d \alpha = 0,$$
and, by assumption, for the total variation \textup{Var} $\mathbb{I}_{L}$ of
$\mathbb{I}_{L}$ on $\left[\left. 0,1\right)\right.$ we have
\begin{equation} \label{equ_b}
\textup{Var} ~\mathbb{I}_{L} \leq 2 A^{L}.
\end{equation}

Furthermore we have
\begin{equation} \label{equ_c}
\left\|\mathbb{I}_{L} \right\|^{2}_{2} := \int^{1}_{0} \left(\mathbb{I}_{L}
\left(\alpha\right)\right)^{2} d \alpha= \mathbb{P}
\left(R_{L}\right)\left(1-\mathbb{P} \left(R_{L}\right)\right) \leq \mathbb{P}
\left(R_{L}\right).
\end{equation}

Let $H_{L} := \left\lfloor \left(1+\eta\right)^{L} \frac{1}{\mathbb{P}
\left(R_{L}\right)}\right\rfloor$. Then by definition
\begin{equation} \label{equ_d}
\sum_{h \leq H_{L}} 1_{L} \left(h \alpha\right) = \sum_{h \leq H_{L}}
\int^{1}_{0} 1_{L} (\omega) d \omega + \sum_{h \leq H_{L}} \mathbb{I}_{L}
\left(h \alpha\right),
\end{equation}
and
\begin{equation} \label{equ_e}
\sum_{h \leq H_{L}} \int^{1}_{0} 1_{L} (\omega) d \omega = H_{L} \mathbb{P}
\left(R_{L}\right) \geq \frac{1}{2} \left(1+\eta\right)^{L}
\end{equation}
for $L$ large enough.\\

We write
$$\mathbb{I}_{L} \left(\alpha\right) \sim \sum^{\infty}_{j=1} \left(u_{j}  \cos
2 \pi j \alpha + v_{j}  \sin 2 \pi j \alpha\right)$$
for the Fourier series of $\mathbb{I}_{L}$. From~\eqref{equ_b} and a classical inequality for the size of the Fourier
coefficients of functions of bounded variation (see for example~\cite{zyg}, p.
48) we have
\begin{eqnarray} \label{equ_f}
& & \left|u_{j}\right| \leq \frac{\textup{Var} ~\mathbb{I}_{L}}{2 j} \leq
\frac{A^{L}}{j} ~\quad \mbox{and}\\ \nonumber
& & \left|v_{j}\right| \leq \frac{A^{L}}{j}. 
\end{eqnarray}

We split the function $\mathbb{I}_L$ into an even and an odd part (that is, into
a cosine- and a sine-series). In the sequel, we consider only the even part; the
odd part can be treated in exactly the same way. Set $G_{L} := \left(A B 
\left(1+\eta\right)\right)^{2L}$, where $A$ and $B$ are the constants from the statement of the theorem. Let $p_L(\alpha)$ denote the $G_{L}$-th
partial sum of the Fourier series of the even part of $\mathbb{I}_L$, and let
$r_L(\alpha)$ denote the remainder term (for simplicity of writing we assume that $G_L$ is an integer). Then by Minkowski's inequality we have
\begin{equation} \label{equ_g}
\left\| \sum_{h=1}^{H_L} \mathbb{I}_L^{(\textup{even})} (h \cdot) \right\|_2
\leq \left\| \sum_{h=1}^{H_L} p_L (h \cdot) \right\|_2 + \left\|
\sum_{h=1}^{H_L} r_L (h \cdot) \right\|_2.
\end{equation}
Furthermore,~\eqref{equ_f}, Minkowski's inequality, and Parseval's identity
imply that
\begin{eqnarray}
\left\|\sum_{h=1}^{H_L} r_L (h \cdot) \right\|_2 & \leq & H_L \|r_L\|_2
\nonumber\\
& \leq & H_L \sqrt{\sum_{j=G_{L} + 1}^\infty \frac{A^{2L}}{j^2}} \nonumber\\
& \leq & \frac{H_L A^L}{G_L^{1/2}} \nonumber\\
& \leq & \frac{(1+\eta)^L A^L}{\p(R_L) G_L^{1/2}} \nonumber\\
& \leq & 1. \label{equ_h}
\end{eqnarray}
To estimate the first term on the right-hand side of~\eqref{equ_g}, we expand
$p_L$ into a Fourier series and use the orthogonality of the trigonometric
system. Then we obtain
\begin{eqnarray}
\left\| \sum_{h=1}^{H_L} p_L (h \cdot) \right\|_2^2 & = &
\underbrace{\sum_{k_1,k_2=1}^{H_L} \sum_{j_1,j_2=1}^{G_{L}}}_{j_1 h_1 = j_2 h_2}
\frac{u_{j_1} u_{j_2}}{2}  \nonumber\\
& = &  \sum_{j_1,j_2=1}^{G_{L}} \frac{u_{j_1} u_{j_2}}{2} ~\# \Big\{ (h_1,
h_2):~1 \leq h_1, h_2 \leq H_L, ~j_1 h_1 = j_2 h_2 \Big\}. \label{equ_i}
\end{eqnarray}

To estimate the size of the sum on the right-hand size of~\eqref{equ_i}, we
assume that $j_1$ and $j_2$ are fixed. It turns out that we have $j_1 h_1 = j_2 h_2$ whenever
$$
h_1 = w \frac{j_2}{\gcd(j_1,j_2)}, \quad h_2 = w \frac{j_1}{\gcd(j_1,j_2)}
\qquad \textrm{for some positive integer $w$}
$$
(see Section 5 of~\cite{AHL} for a more detailed deduction). As a consequence we have
\begin{eqnarray}
& & \# \Big\{ (h_1, h_2):~1 \leq h_1, h_2 \leq H_L, ~j_1 h_1 = j_2 h_2 \Big\}
\nonumber\\
& = & \# \left\{w \geq 1:~ w \leq \min \left( \frac{H_L \gcd(j_1,j_2)}{j_2},
\frac{H_L \gcd(j_1,j_2)}{j_1} \right) \right\} \nonumber\\
& = & \left\lfloor \frac{H_L \gcd(j_1,j_2)}{\max(j_1,j_2)} \right\rfloor
\nonumber\\
& \leq & \frac{H_L \gcd(j_1,j_2)}{\sqrt{j_1 j_2}}. \nonumber
\end{eqnarray}

Combining this estimate with~\eqref{equ_i} we obtain
\begin{equation} \label{equ_j}
\left\| \sum_{h=1}^{H_L} p_L (h \cdot) \right\|_2^2 \leq H_L
\sum_{j_1,j_2=1}^{G_{L}} \frac{|u_{j_1} u_{j_2}|}{2}
\frac{\gcd(j_1,j_2)}{\sqrt{j_1 j_2}}.
\end{equation}
The sum on the right-hand side of the last equation is called a \emph{GCD sum}.
It is well-known that such sums play an important role in the metric theory of
Diophantine approximation; the particular sum in~\eqref{equ_j} probably appeared
for the first time in LeVeque's paper~\cite{LeVe} (see also~\cite{abs}
and~\cite{dhs}). A precise upper bound for these sums has been obtained by
Hilberdink~\cite{hilber}.\footnote{The upper bounds for the GCD sums
in~\cite{hilber} are formulated in terms of the largest eigenvalues of certain
GCD matrices; since these matrices are symmetric and positive definite, the
largest eigenvalue also gives an upper bound for the GCD sum. This relation is
explained in detail in~\cite{absw}.} Hilberdink's result implies that there
exists an absolute constant $c_{\textup{abs}}$ such that
$$
\sum_{j_1,j_2=1}^{G_{L}} \frac{|u_{j_1} u_{j_2}|}{2}
\frac{\gcd(j_1,j_2)}{\sqrt{j_1 j_2}} \ll \exp \left( \frac{c_{\textup{abs}}
\sqrt{\log (G_{L})}}{\sqrt{\log \log G_{L}}} \right) \sum_{j=1}^{G_{L}} u_j^2.
$$

Combining this estimate with~\eqref{equ_c} and~\eqref{equ_j} (and using
Parseval's identity) we have
\begin{eqnarray*}
\left\| \sum_{h=1}^{H_L} p_L (h \cdot) \right\|_2^2 & \ll & H_L \exp \left(
\frac{c_{\textup{abs}} \sqrt{\log (G_{L})}}{\sqrt{\log \log G_{L}}} \right)
\mathbb{P} \left(R_{L}\right) \\
& \ll & (1+\eta)^L \exp \left( \frac{c_{\textup{abs}} \sqrt{\log
(G_{L})}}{\sqrt{\log \log G_{L}}} \right),
\end{eqnarray*}
and, together with~\eqref{equ_g} and~\eqref{equ_h}, and with a similar argument
for the odd part of $\mathbb{I}_L$, we obtain
\begin{equation} \label{equ_k}
\left\| \sum_{h=1}^{H_L} \mathbb{I}_L (h \cdot) \right\|_2^2 \ll (1+\eta)^L \exp
\left( \frac{c_{\textup{abs}} \sqrt{\log (G_{L})}}{\sqrt{\log \log G_{L}}}
\right).
\end{equation}
By Chebyshev's inequality we have
\begin{eqnarray*}
\mathbb{P} \left( \alpha \in [0,1):~ \left| \sum_{h=1}^{H_L} \mathbb{I}_L (h
\alpha) \right| > (\log H_L) \left\| \sum_{h=1}^{H_L} \mathbb{I}_L (h \cdot)
\right\|_2 \right) & \leq & \frac{1}{(\log H_L)^2},
\end{eqnarray*}
and since $(H_L)_{L \geq 1}$ grows exponentially in $L$ these probabilities give
a convergent series when summing over $L$. Thus by the Borel--Cantelli lemma
for almost all $\alpha$ only finitely many events
$$
\left| \sum_{h=1}^{H_L} \mathbb{I}_L (h \alpha) \right| > (\log H_L) \left\|
\sum_{h=1}^{M} \mathbb{I}_L (h \cdot) \right\|_2
$$
happen, which by~\eqref{equ_k} implies that 
$$
\left| \sum_{h=1}^{H_L} \mathbb{I}_L (h \alpha) \right| \ll (1 + \eta)^{L/2}
\exp \left( \frac{\hat{c}_{\textup{abs}} \sqrt{L}}{\sqrt{\log L}} \right)
$$
for some absolute constant $\hat{c}_{\textup{abs}}$. Comparing this upper bound
with~\eqref{equ_d} and using~\eqref{equ_e} we conclude that
$$
\sum_{h=1}^{H_L}  \mathbf{1}_L (h \alpha) \gg (1+\eta)^L \qquad \textrm{as $L
\to \infty$}
$$
for almost all $\alpha$. In particular we have
$$
\sum_{L=1}^\infty \sum_{h=1}^{H_L}  \mathbf{1}_L (h \alpha) = \infty
$$
for almost all $\alpha$, and the result follows.

\end{proof}

\begin{proof}[Proof of Theorem~\ref{th_3}]

By the Koksma-Hlawka inequality (see for example~\cite{dts} or ~\cite{knu}) for every positive
integer $H$ we have
\begin{equation} \label{co1*}
N D_{N} \geq \frac{1}{4 H} \left|\sum^{N}_{n=1} e^{2 \pi i H a_{n}
\alpha}\right|.
\end{equation}

Let 
$$f_{L} \left(\alpha\right) := \left|\sum^{B_{L}}_{n=1} e^{2 \pi i a_{n} \alpha}
\right|.$$
We will show that for any given $\varepsilon >0$ for almost all $\alpha$
there are infinitely many $L$ such that there exists a positive integer $h_{L}$
with
\begin{equation} \label{equ_l}
\frac{1}{h_{L}} f_{L} \left(h_{L} \alpha\right) \gg B^{\tau-\varepsilon}_{L},
\end{equation}
which together with~\eqref{co1*} implies the conclusion of the theorem.\\

Let $\varepsilon >0$ with $\varepsilon < \frac{\tau}{2}$ be given. From the
definition of $f_{L} \left(\alpha\right)$ and the fact that $|a_n| \leq n^t$ for large $n$ it is easily seen that for sufficiently large $L$ we have
\begin{eqnarray}
\left|f_{L} \left(\alpha_{1}\right) - f_{L} \left(\alpha_{2}\right)\right| &
\leq & 2 \pi B_{L} B_{L}^{t} \left|\alpha_{1} -\alpha_{2}\right| \nonumber\\
& \leq & 2 \pi \left(B^{1+t}\right)^{L} \left|\alpha_{1} -\alpha_{2}\right|.
\label{BL*}
\end{eqnarray}

Now let $g_{L} \left(\alpha\right)$ be the function defined by 
\begin{eqnarray*}
g_{L} \left(\alpha\right):= f_{L} \left(j \left(B^{1+t}\right)^{-L}\right) & &
\mbox{for} \quad \alpha \in \left[\left. j \left(B^{1+t}\right)^{-L},
\left(j+1\right) \left(B^{1+t}\right)^{-L}\right)\right.\\
& & \mbox{and for} \quad j=0,1, \ldots, \left(B^{1+t}\right)^{L} -1
\end{eqnarray*}
(for simplicity of writing we assume that $\left(B^{1+t}\right)^{L}$ is an
integer). Then by~\eqref{BL*} we have
$$\left|g_{L} -f_{L}\right| \leq 2 \pi,$$
which means that it is sufficient to prove~\eqref{equ_l} with $f_{L}$ replaced
by $g_{L}$. By construction the function $g_{L}$ can be written as a sum of at
most $\left(B^{1+t}\right)^{L}$ indicator functions of disjoint intervals.\\

Let
$$Q:= \left\lfloor\frac{1-\tau+2 \varepsilon}{3 \varepsilon}\right\rfloor +1$$
and
$$\Delta_{i} := B_{L}^{\tau-2 \varepsilon} B_{L}^{3 \varepsilon i} \quad
\mbox{for}~ i=0,1,\ldots,Q.$$

Note that $\Delta_{0}=B_{L}^{\tau-2 \varepsilon}$ and $B_{L} \leq \Delta_{Q}
\leq B_{L}^{1+3 \varepsilon}$. Define

$$
M_{L}^{i} := \Big\{\alpha \in \left[\left.0,1\right.\right)~\mbox{with}~
\Delta_{i} < \left|g_{L} \left(\alpha\right)\right| \leq  \Delta_{i+1}\Big\}
$$
for $i=0, \ldots, Q-1.$ Then by~\eqref{equ_a} we have
\begin{eqnarray*}
\sum^{Q-1}_{i=0}\mathbb{P} \left(M_{L}^{\left(i\right)}\right) \Delta_{i+1} +
\left(1-\sum^{Q-1}_{i=0} \mathbb{P} \left(M_{L}^{(i)}\right)\right) \Delta_{0} &
\geq & \int^{1}_{0} \left|g_{L} \left(\alpha\right)\right| d \alpha \\
& > & \frac{B^{\tau- \varepsilon}_{L}}{2}
\end{eqnarray*}
for sufficiently large $L$.\\

By $\Delta_{0}=B_{L}^{\tau-2 \varepsilon}$ we have
$$\sum^{Q-1}_{i=0} \mathbb{P} \left(M_{L}^{(i)}\right) \Delta_{i+1} \geq
\frac{B_{L}^{\tau-\varepsilon}}{4}$$

for $L$ large enough. Consequently for every $L$ large enough there is an $i_{L}
\in \left\{0, \ldots, Q-1 \right\}$ with 

$$\Delta_{i_{L}+1} \mathbb{P} \left(M_{L}^{(i_{L})}\right) \geq
\frac{B_{L}^{\tau-\varepsilon}}{4Q},$$
which implies that
$$\mathbb{P} \left(M_{L}^{(i_{L})}\right) \geq \frac{1}{4Q} \frac{B_{L}^{\tau-
\varepsilon}}{\Delta_{i_{L}+1}} \geq \frac{1}{4Q} \frac{1}{B_{L}^{1-\tau+4
\varepsilon}} \geq \left(\frac{1}{B^{1-\tau+5 \varepsilon}}\right)^{L}$$
for $L$ large enough. Note that, as a consequence of the construction of
$g_{L}$, the set $M_{L}^{\left(i_{L}\right)}$ is always a union of at most
$\left(B^{1+t}\right)^{L}$ intervals. By Theorem~\ref{th_4} we conclude that for
almost all $\alpha$ for all $\eta>0$ there are infinitely many integers $h_{L}$
with $h_{L} < \left(1+\eta\right)^{L} \frac{1}{\mathbb{P}
\left(M_{L}^{\left(i_{L}\right)}\right)}$ and $\left\{h_{L} \alpha\right\} \in
M_{L}^{(i_{L})}$. Consequently for almost all $\alpha$ we have
\begin{eqnarray*}
\frac{1}{h_{L}} \left|g_{L} \left(h_{L} \alpha\right)\right| & \geq &
\frac{1}{\left(1+\eta\right)^{L}} \mathbb{P}
\left(M_{L}^{\left(i_{L}\right)}\right) \Delta_{i_{L}}\\
& \geq & \frac{1}{\left(1+\eta\right)^{L}}
\frac{\Delta_{i_{L}}}{\Delta_{i_{L}+1}} \Delta_{i_{L}+1} \mathbb{P}
\left(M_{L}^{(i_{L})}\right) \\
& \geq & \frac{1}{\left(1+\eta\right)^{L}} \frac{1}{B_{L}^{3 \varepsilon}}
\frac{B_{L}^{\tau- \varepsilon}}{4Q}\\
& \geq & B_{L}^{\tau- 5 \varepsilon}
\end{eqnarray*}
for $\eta$ small enough and for infinitely many $L$, since by assumption $B_{L}
\geq \left(B'\right)^{L}$ for some constant $B' > 1$. This proves the theorem.
\end{proof}

\section{Proofs of Theorem~\ref{th_2} and Theorem~\ref{th_1}} \label{sect_3}

\begin{proof}[Proof of Theorem~\ref{th_2}]
By Theorem~\ref{th_3} it suffices to show that
$$
\int^{1}_{0} \left|\sum^{N}_{n=1} e^{2 \pi i f(n) \alpha}\right|d \alpha >
N^{\frac{1}{2}-\varepsilon}
$$
for all $\varepsilon > 0$ and $N$ large enough.\\

Using a standard trick, by H\"older's inequality with
$$
F\left(\alpha\right):= \left|\sum^{N}_{n=1} e^{2 \pi i f(n)
\alpha}\right|^{\frac{2}{3}} \qquad \textrm{and} \qquad G\left(\alpha\right) := \left|\sum^{N}_{n=1} e^{2 \pi i f(n)
\alpha}\right|^{\frac{4}{3}}
$$ 
and with $p=\frac{3}{2}$ and $q=3$ we have
\begin{eqnarray*}
\int^{1}_{0} \left|\sum^{N}_{n=1} e^{2 \pi i f(n) \alpha}\right|^{2} d \alpha &
= & \int^{1}_{0} F\left(\alpha\right) G\left(\alpha\right) d\alpha \\
& \leq & \left(\int^{1}_{0} \left(F\left(\alpha\right)\right)^{\frac{3}{2}} d
\alpha \right)^{\frac{2}{3}} \left(\int^{1}_{0}
\left(G\left(\alpha\right)\right)^{3} d \alpha\right)^{\frac{1}{3}}\\
& = & \left(\int^{1}_{0} \left|\sum^{N}_{n=1} e^{2 \pi i f(n) \alpha}\right|
d\alpha \right)^{\frac{2}{3}} \left(\int^{1}_{0} \left|\sum^{N}_{n=1} e^{2 \pi i
f(n) \alpha}\right|^{4} d \alpha\right)^{\frac{1}{3}},
\end{eqnarray*}
hence
\begin{equation} \label{124}
\int^{1}_{0} \left|\sum^{N}_{n=1} e^{2 \pi i f(n) \alpha} \right| d \alpha \geq
\frac{\left(\left|\sum^{N}_{n=1} e^{2 \pi i f(n) \alpha}\right|^{2} d
\alpha\right)^{\frac{3}{2}}}{\left(\left|\sum^{N}_{n=1} e^{2 \pi i f(n)
\alpha}\right|^{4} d \alpha \right)^{\frac{1}{2}}}.
\end{equation}
Now
$$
\int^{1}_{0} \left|\sum^{N}_{n=1} e^{2 \pi i f(n) \alpha}\right|^{2} d \alpha =
\sum_{\substack{1 \leq m,n \leq N,\\f(m) = f(n)}} 1 \geq N,
$$
and
\begin{eqnarray*}
\int^{1}_{0} \left|\sum^{N}_{h=1} e^{2 \pi i f(n) \alpha}\right|^{4} d \alpha &
= & \underset{f(k) + f(l) = f(m) +f(n)}{\sum_{1 \leq k,l,m,n \leq N,}} 1\\
& = & \sum_{a} A^{2}_{f} (a) = \sum_{a} \tilde{A}^{2}_{f} (a),
\end{eqnarray*}
where summation in the penultimate sum is over all integers $a \in [-2 N^{t}, 2 N^{t}]$ such
that there exist $k,l$ with $1 \leq k,l \leq N$ and $f(k) + f(l) = a$, and summation in the last sum is over all $a \in [-2 N^{t},2 N^{t}]$ such that there exist $k, l$ with $1 \leq k, l \leq N$ and $f(k) -f(m) = a$. By assumption we either have
\begin{eqnarray*}
\sum_{a} A_{f}^{2} (a) & \leq & \left( \underset{a \in [-2N^t,2N^t]}{\max}~\left\{  A^{2}_{f} (a) \right\} \right) \# \left\{a:~ \exists ~ 1 \leq k,l \leq N
~\mbox{with}~f(k) + f(l) =a\right\}\\
& \leq & \left(2 N^{t}\right)^{2 \frac{\varepsilon}{2 t}} N^{2}\\
& \leq & c(\varepsilon) N^{2+ \varepsilon}
\end{eqnarray*}
for all $\varepsilon >0$ and all $N\geq N(\varepsilon)$, or we have the corresponding estimate for $\sum_{a} \tilde{A}^{2}_{f} (a)$.\\

Hence by~\eqref{124} for all $\varepsilon > 0$ and for $N \geq N\left(\varepsilon\right)$ we have
$$
\int^{1}_{0} \left|\sum^{N}_{n=1} e^{2 \pi i f(n) \alpha}\right| d \alpha \geq
\frac{N^{\frac{3}{2}}}{\left(c\left(\varepsilon\right) N^{2+
\varepsilon}\right)^{\frac{1}{2}}} =
\frac{1}{c\left(\varepsilon\right)^{\frac{1}{2}}}
N^{\frac{1}{2}-\frac{\varepsilon}{2}}
$$
and the result follows.
\end{proof}

\begin{proof}[Proof of Theorem~\ref{th_1}]
By Theorem~\ref{th_2} it suffices to show that $f(n) := P\left(m_{n}\right)$
satisfies 
$$
\tilde{A}_{f} (n) = \mathcal{O} \left(|n|^{\varepsilon}\right)
$$
for all $\varepsilon >0$.\\

Let first $m_{n} =n$. By assumption $f$ is of degree $d \geq 2$, hence there exists a non-constant polynomial $q \in \mathbb{Z}\left[x,y\right]$ such that $f(x) -f(y) = (x-y) q(x,y)$. So, if $f(x) -f(y) = n$ for some non-zero integer $n$, it follows that $x-y$ is a divisor $t$ of $n$, hence $q\left(x, x-t\right) = \frac{n}{t}$. This last equation has at most $d-1$ solutions $x$. Since $n$ has $\mathcal{O}\left(\left|n\right|^{\varepsilon}\right)$ divisors $t$ for all $\varepsilon > 0$, the assertion $\tilde{A}_{f} (n) = \mathcal{O}\left(\left|n\right|^{\varepsilon}\right)$ follows.\\

For arbitrary $\left(m_{n}\right)$ the result follows trivially from the special
result for $m_{n} =n$.
\end{proof}

\section{Proof of Theorem~\ref{th_5}} \label{sect_4}

For the proof of Theorem~\ref{th_5} we will make use of the following Lemma on
continued fractions.

\begin{lemma}
For $x \in \mathbb{R}$ let $b_{m} (x)$ denote the $m$-th continued
fraction coefficient of $x$. For any integer $\beta \geq 2$ and $L \in \mathbb{N}$
let
$$
S_{L} (\alpha):= \sum^{L}_{k=1} \sum^{L}_{m=1} b_{m} \left(\beta^{k}
\alpha\right).
$$
Then for almost all $\alpha$ we have 
$$
S_{L} (\alpha) = \mathcal{O} \left(L^{2+\varepsilon}\right)
$$
for every $\varepsilon >0$.
\end{lemma}

\begin{proof}
This is a special case of Lemma~2 in~\cite{larch}.
\end{proof}

\begin{proof}[Proof of Theorem~\ref{th_5}]
For given integer $d \geq 1$ we now construct a sequence $\left(b_{n}\right)_{n
\geq 1}$ which satisfies the properties stated in Theorem~\ref{th_5}.\\

Let $\left(b_{n}\right)_{n \geq1}$ be the strictly increasing sequence of
integers running through the integers
$$
2^{d k} + j 2^{d k +d-k}
$$
for $j=0,\ldots, 2^{k}-1$ and $k=1,2,3,\ldots$.\\

We have
\begin{eqnarray*}
\frac{b_{n+1}}{b_{n}} & = & \frac{2^{d k}+j 2^{d k+d-k}+2^{d k+d-k}}{2^{d k}+j
2^{d k +d-k}} \\
& = & 1+ \frac{1}{2^{k-d}+j}\\
& \geq & 1 + \frac{1}{\left(2^{d k}+ j2^{d k +d-k}\right)^{\frac{1}{d}}}\\
& = & 1 + \frac{1}{{b_{n}}^{\frac{1}{d}}},
\end{eqnarray*}
since (as is easily checked)
\begin{eqnarray*}
2^{k-d}+j & \leq &\left(2^{d k} + j 2^{d k +d-k}\right)^{\frac{1}{d}} 
\end{eqnarray*}
for $j=0$ and $j=2^{k}$ and hence for all $j=0, \ldots, 2^{k}$. So
$\left(b_{n}\right)_{n \geq 1}$ has polynomial growth behavior of degree at
least $d$.\\

Let $N$ be given and $k_{0}$ and $j_{0}$ such that $b_{N} = 2^{d k_{0}}+j_{0}
2^{d k_{0} +d -k_{0}}$ for some $j_{0} \in \left\{0, \ldots,
2^{k_{0}}-1\right\}.$ By standard techniques from the theory of uniform
distribution theory for the discrepancy $D_{N}$ of $\left(\left\{b_{n}
\alpha\right\}\right)_{n \geq 1}$ we have
$$
ND_{N} \leq \sum^{k_{0}}_{k=1} 2^{k} D^{(k)},
$$
where by $D^{(k)}$ we denote the discrepancy of the subsequence
$$
\left(\left\{2^{d k} \alpha + j 2^{\left(d-1\right) k} 2^{d}
\alpha\right\}\right); \quad j=0,1,\ldots, 2^{k}-1.
$$
For this sequence it is also well known (see for example~\cite{knu}) that
$$2^{k} D^{(k)} \leq c_{\textup{abs}} \sum^{2 k}_{m=1} b_{m}
\left(\left(2^{(d-1)}\right)^{k} 2^{d} \alpha\right),$$
with an absolute constant $c_{\textup{abs}}$, hence
$$ND_{N} \leq c_{\textup{abs}} \sum^{2 k_{0}}_{k=1} \sum^{2 k_{0}}_{m=1} b_{m}
\left(\beta^{k} 2^{d} \alpha \right),$$
where $\beta=2^{\left(d-1\right)}$ and $b_{m} (x)$ again denotes the
$m$-th continued fraction coefficient of $x$. By Lemma 1 therefore for
almost all $2^{d} \alpha$ (and consequently also for almost all $\alpha$) we have
$$
ND_{N} = \mathcal{O} \left(k_{0}^{2+\varepsilon}\right) = \mathcal{O}
\left(\left(\log N\right)^{2 + \varepsilon}\right) \qquad \textrm{as} \quad N \to \infty
$$
for all $\varepsilon >0$ (note that $N \geq 2^{k_{0}}$). This proves the
theorem.
\end{proof}


\begin{thebibliography}{10}

\bibitem{abs}
C.~Aistleitner, I.~Berkes, and K.~Seip.
\newblock {G}{C}{D} sums from {P}oisson integrals and systems of dilated
  functions.
\newblock \emph{J. Eur. Math. Soc.}, 17(6):1517--1546, 2015.

\bibitem{absw}
C.~Aistleitner, I.~Berkes, K.~Seip, and M.~Weber.
\newblock Convergence of series of dilated functions and spectral norms of
  {G}{C}{D} matrices.
\newblock \emph{Acta Arith.}, 168(3):221-246, 2015.

\bibitem{AHL}
C.~Aistleitner, R.~Hofer and G.~Larcher.
\newblock On parametric Thue-Morse Sequences and Lacunary Trigonometric
Products.
\newblock Submitted, 2015. Available at
  \url{http://arxiv.org/abs/1502.06738}.

\bibitem{baker}
R.~C. Baker.
\newblock Metric number theory and the large sieve.
\newblock {\em J. London Math. Soc. (2)}, 24(1):34--40, 1981.

\bibitem{bpt}
I.~Berkes and W.~Philipp.
\newblock {\em J. London Math. Soc. (2)}, 50(3):454--464, 1994.

\bibitem{dts}
M.~Drmota and R.~F. Tichy.
\newblock {\em Sequences, discrepancies and applications}, volume 1651 of {\em
  Lecture Notes in Mathematics}.
\newblock Springer-Verlag, Berlin, 1997.

\bibitem{dhs}
T.~Dyer and G.~Harman.
\newblock Sums involving common divisors.
\newblock {\em J. London Math. Soc. (2)}, 34(1):1--11, 1986.

\bibitem{fjk}
H.~Fiedler, W.~Jurkat, and O. ~K\"orner.
\newblock Asymptotic expansions of finite theta series.
\newblock \emph{Acta Arith.}, 32:129--146, 1977.

\bibitem{foumau}
E.~Fouvry and C.~Mauduit.
\newblock Sommes des chiffres et nombres presque premiers.
\newblock {\em Math. Ann.}, 305(3):571--599, 1996.

\bibitem{fu2}
K.~Fukuyama.
\newblock The law of the iterated logarithm for discrepancies of
{$\{\theta^nx\}$}.
\newblock {\em Acta Math. Hungar.}, 118(1-2):155--170, 2008.


\bibitem{hilber}
T.~Hilberdink.
\newblock An arithmetical mapping and applications to {$\Omega$}-results for
  the {R}iemann zeta function.
\newblock {\em Acta Arith.}, 139(4):341--367, 2009.

\bibitem{khin}
A.~{Khintchine}.
\newblock {Einige {S}\"atze \"uber {K}ettenbr\"uche, mit {A}nwendungen auf die
  {T}heorie der {D}iophantischen {A}pproximationen.}
\newblock {\em {Math. Ann.}}, 92:115--125, 1924.

\bibitem{knu}
L.~Kuipers and H.~Niederreiter.
\newblock {\em Uniform distribution of sequences}.
\newblock Wiley-Interscience [John Wiley \& Sons], New York-London-Sydney,
  1974.

\bibitem{larch}
G.~Larcher.
\newblock Probabilistic {D}iophantine approximation and the distribution of
  {H}alton--{K}ronecker sequences.
\newblock {\em J. Complexity}, 29:397--423, 2013.

\bibitem{LeVe}
W.~J. LeVeque.
\newblock On the frequency of small fractional parts in certain real sequences
  {I}{I}{I}.
\newblock {\em J. Reine Angew. Math.}, 202:215--220, 1959.


\bibitem{philipp}
W.~Philipp.
\newblock Limit theorems for lacunary series and uniform distribution {${\rm
  mod}\ 1$}.
\newblock {\em Acta Arith.}, 26(3):241--251, 1974/75.

\bibitem{wey}
H.~Weyl.
\newblock \"Uber die Gleichverteilung von Zahlen modulo Eins.
\newblock {\em Math. Ann.}, 77:313--352, 1916.

\bibitem{zyg}
A.~Zygmund.
\newblock {\em Trigonometric series. {V}ol. {I}, {II}}.
\newblock Cambridge Mathematical Library. Cambridge University Press,
  Cambridge, 1988.
\newblock Reprint of the 1979 edition.


\end{thebibliography}
\end{document}